\documentclass{amsart}

\usepackage{amsbsy}
\usepackage{amsmath}
\usepackage{amsfonts}
\usepackage{braket}
\usepackage{latexsym}
\usepackage{amsthm}
\usepackage{amsxtra}
\usepackage{amssymb}
\usepackage{mathrsfs}
\usepackage{verbatim}  

\usepackage[usenames,dvipsnames]{xcolor}

\usepackage[shortlabels]{enumitem}
\SetEnumerateShortLabel{a}{\textup{(\alph*)}}
\SetEnumerateShortLabel{A}{\textup{(\Alph*)}}
\SetEnumerateShortLabel{1}{\textup{(\arabic*)}}
\SetEnumerateShortLabel{i}{\textup{(\roman*)}}
\SetEnumerateShortLabel{I}{\textup{(\Roman*)}}

\usepackage{hyperref} 
\hypersetup{breaklinks, raiselinks}
\hypersetup{backref}
\hypersetup{colorlinks, citecolor=ForestGreen, linkcolor=ForestGreen, urlcolor=ForestGreen}
\hypersetup{pagebackref}

\usepackage[initials,lite]{amsrefs} 

\usepackage{a4wide}

\theoremstyle{plain}
\newtheorem{thm}{Theorem}[section]
\newtheorem{Theorem}[thm]{Theorem}
\newtheorem{theorem}[thm]{Theorem}

\newtheorem{lemma}[thm]{Lemma}

\newtheorem*{proper*}{Property}

\newtheorem{cor}[thm]{Corollary}

\newtheorem{corollary}[thm]{Corollary}

\newtheorem{conjecture}[thm]{Conjecture}

\newtheorem*{lm*}{Lemma}
\newtheorem*{thm*}{Theorem}
\newtheorem*{cor*}{Corollary}

\theoremstyle{definition}

\newtheorem{definition}[thm]{Definition}
\newtheorem*{df*}{Definition}

\newtheorem{ex-notn}[thm]{Example/Notation}

\newtheorem{construction}[thm]{Construction}
\newtheorem{observation}[thm]{Observation}

\theoremstyle{remark}

\newtheorem{remark}[thm]{Remark}

\newtheorem*{acknowledgement*}{Acknowledgement}

\newtheorem*{ex*}{Example}
\newtheorem*{exer*}{Exercise}
\newtheorem*{rem*}{Remark}
\newtheorem*{prob*}{Problem}
\newtheorem*{prop*}{Proposition}

\newtheoremstyle{mystyle}{5pt}{5pt}%
{\itshape}
{}
{\bfseries}
{.}
{.5em}
{\thmname{#1}\thmnumber{ #2}\thmnote{ #3}}

\theoremstyle{mystyle}
\newtheorem*{thma}{Theorem}

\def\Ass{\operatorname{Ass}}

\def\boxicity{\operatorname{box}}

\def\cochord{\operatorname{cochord}} 

\def\cosize{\operatorname{cosize}}

\def\depth{\operatorname{depth}}

\def\fdepth{\operatorname{fdepth}} 

\def\Ht{\operatorname{ht}}

\def\indmatch{\operatorname{indmatch}} 

\def\is{\operatorname{is}} 

\def\lcm{\operatorname{lcm}}

\def\reg{\operatorname{reg}}

\def\projdim{\operatorname{proj\,dim}}

\def\Red{\mathrm{red}}

\def\sdepth{\operatorname{sdepth}} 

\def\size{\operatorname{size}}

\def\sreg{\operatorname{sreg}}

\def\supp{\operatorname{supp}}

\def\KK{{\mathbb K}}

\def\NN{{\mathbb N}}

\def\ZZ{{\mathbb Z}}

\def\calC{\mathcal{C}}
\def\calD{\mathcal{D}}

\def\calF{\mathcal{F}}

\usepackage{bm}

\def\bda{{\bm a}}

\def\bdx{{\bm x}}

\def\alert#1{\textcolor{Magenta}{#1}}
\def\ceil#1{\left\lceil #1 \right\rceil}
\def\deg{\operatorname{deg}}

\def\floor#1{\left\lfloor #1 \right\rfloor}
\def\Index#1{\emph{#1}}

\def\isom{\cong}

\def\mySet#1{\left\{ #1\right \}}

\begin{document}
\title[Bounding {S}tanley depth and {S}tanley regularity]{Bounds on the {S}tanley depth and {S}tanley regularity of edge ideals of clutters}
\author{Yi-Huang Shen}
\address{Wu Wen-Tsun Key Laboratory of Mathematics of CAS and School of Mathematical Sciences, University of Science and Technology of China, Hefei, Anhui, 230026, People's Republic of China}
\thanks{This work is supported by the National Natural Science Foundation of China (11201445). We thank Hailong Dao and Jay Schweig for helpful explanation during the preparation of the current work.}
\email{yhshen@ustc.edu.cn}
\date{\today}
\begin{abstract}
  Let $I$ be the edge ideal of a clutter $\calC$ in a polynomial ring $S$.  In this paper, we present estimations of the Stanley depth of $I$ as well as the Stanley regularity of $S/I$, in terms of combinatorial data from the clutter $\calC$.
\end{abstract}
\maketitle

\section{Introduction}
Depth, projective dimension and Castelnuovo-Mumford regularity are three important and closely related invariants in commutative algebra and algebraic geometry. For example, if $S=\KK[x_1,\dots,x_n]$ is a polynomial ring over a field $\KK$ and $I\subset S$ is a monomial ideal, then thanks to Auslander and Buchsbaum \cite[Theorem 1.3.3]{MR1251956} we know that
\[
\depth(S/I)+\projdim(S/I)=n.
\]
If in addition $I$ is squarefree, then $I$ has an Alexander dual $I^\vee$ which is also a squarefree monomial ideal. Now a result of Terai \cite[Corollary 0.3]{MR1715588} asserts that
\[
\projdim(I)=\reg(S/I^\vee).
\]

There are numerous works trying to compute or estimate these three invariants. For instance Lyubeznik considered the size of monomial ideals in the article \cite{MR0921965}. Let $I=\bigcap_{i=1}^s Q_i$ be an irredundant primary decomposition of a monomial ideal $I$ in $S$, where the $Q_i$'s are also monomial ideals.  The \Index{size} of $I$, denoted by $\size(I)$, is the number $v+n-h-1$, where $v$ is the minimal number $t$ such that there exist $j_1<\cdots < j_t$ with $\sqrt{\sum_{k=1}^t Q_{j_k}} = \sqrt{\sum_{j=1}^s Q_j}$, and where $h=\Ht \sum_{j=1}^s Q_j$.  
Lyubeznik \cite[Proposition 2]{MR0921965} acquired that
\begin{equation}
  \depth(S/I)\ge \size(I)
  \label{Lyubeznik-0}
\end{equation}
 and consequently
\begin{equation}
  \depth(I)\ge \size(I)+1. \label{Lyubeznik-1}
\end{equation}

A related result, due to Herzog, Popescu and Vladoiu \cite[Theorem 3.1]{arXiv:1011.6462}, asserts that
\begin{equation}
  \sdepth(I)\ge \size(I)+1, \label{HPV}
\end{equation}
where $\sdepth(I)$ is the Stanley depth of $I$. We will explain the notion of Stanley depth in the next section.

It is conjectured by Stanley \cite{MR666158} that
\begin{equation}
  \sdepth(M)\ge \depth(M) \label{StanleyConjecture} 
\end{equation}
for all finitely generated $\ZZ^n$-graded $S$-module. Obviously, Stanley's conjecture \eqref{StanleyConjecture} for $M=I$ with Lyubeznik's result \eqref{Lyubeznik-1} implies the inequality \eqref{HPV}.

It is worth mentioning that squarefree monomial ideals can be naturally related to clutters. Among many others, recent work of Dao and Schweig \cite{arXiv:1301.2665}, H\`a and Woodroofe \cite{arXiv:1301.6779}, Lin and McCullough \cite{arXiv:1211.4301} and Woodroofe \cite{arXiv:1009.2756} provided several very nice bounds for estimating depth and Castelnuovo-Mumford regularity. These research all involve considerations of combinatorial data from the clutter associated to the squarefree monomial ideals. Thus, analogous to \cite{arXiv:1011.6462}, it is natural to ask whether the above work can find counterparts when estimating Stanley depth and Stanley regularity?

Now it is time to outline the structure of our paper. In Section 2, we provide preliminary background for the notions like clutter, filtration depth, Stanley depth and Stanley regularity respectively.

In section 3, we will study the method of Dao and Schweig \cite{arXiv:1301.2665}, and provide a lower bound of the Stanley depth of squarefree monomials $I$ in terms of the index of edge domination of the associated clutter.  If $\calC$ is a clutter, a collection $F$ of edges in $\calC$ is called \Index{edgewise dominant} if for every vertex $v\in V(\calC^\Red)$ which is not contained in some edge of $F$ or contained in a trivial edge, it has a neighbor contained in some edge of $F$.  The \Index{index of edgewise domination} is the number
  \[
  \epsilon(\calC)=\min\mySet{|F|:F\subset E(\calC)\text{ is edgewise dominant}}.
  \]
Our first main result is

\begin{thma}
 [{\ref{main-ied}}]
  Let $\calC$ be a clutter and $I(\calC)$ the corresponding edge ideal in $S$. Then
  \[
  \min\Set{\depth(S/I(\calC)),\ \sdepth(S/I(\calC))}\ge \epsilon(\calC)+n-|V(\calC^\Red)|.
  \]
\end{thma}

In section 4,  we employ the splitting method in Herzog, Popescu and Vladoiu \cite{arXiv:1011.6462}  and estimate the Stanley regularity of $S/I$. This work is related to the article \cite{arXiv:1211.4301} by Lin and McCullough. We show that 

\begin{thma}
  [{\ref{LM}}]
  Let $\calC=(V,E)$ be a clutter and $\calC'=(V,E')$ be the clutter obtained by removing all edges with free vertices from $\calC$. Let $\beta(\calC')$ be the matching number of $\calC'$. Then
  \[
  \sreg(S/I(\calC))\le |V|-|E|+|E'|-\beta(\calC').
  \]
\end{thma}

In Section 5, we will start by establishing a key result that is similar to the Castelnuovo-Mumford regularity version by Kalai and Meshulam. We will apply it to give various upper bounds for the Stanley regularity of $S/I$ in terms of various packing invariants of the associated clutter.

To be more specific, we will study the notion of co-chordal cover number of a simple graph $G$, which is the minimum number of co-chordal subgraphs required to cover $G$. Similar to a result by Woodroofe \cite{arXiv:1009.2756}, we assert that

\begin{thma}
 [{\ref{cochord-thm}}]
For any simple graph $G$, we have $\sreg(S/I(G))\le \cochord(G)$.
\end{thma}

Our last result is based on the notion of 2-collage, introduced by H\`a and Woodroofe in \cite{arXiv:1301.6779}. Let $\calC$ be a clutter.  Then a \Index{2-collage} for $\calC$ is a subset $C$ of edges with the property that for each $E\in E(\calC)$ we can delete a vertex $v$ so that $E\setminus\Set{v}$ is contained in some edge of $C$. We claim that 

\begin{thma}
  [{\ref{collages}}]
  If $\Set{E_1,\dots,E_s}$ is a $2$-collage in the clutter $\calC$, then 
  \[
  \sreg(S/I(\calC))\le \sum_{i=1}^s (|E_i|-1).
  \]
\end{thma}

Here is the final comment before we start a new section. The Stanley's conjecture \eqref{StanleyConjecture} is still widely open so far. This happens partly due to the lack of powerful tools like long exact sequence and depth lemma \cite[Proposition 1.2.9]{MR1251956}. What we have so far that is most similar to the depth lemma is as follows: let
$0\to M\to N\to L\to 0$ be a short exact sequence of finitely generated $\ZZ^n$-graded $S$-modules, then
\begin{equation}
  \label{middle-sdepth}
  \sdepth(N)\ge \min\Set{\sdepth(M),\ \sdepth(L)}
\end{equation}
by \cite[Proposition 2.6]{bruns-2009}.
Since the research on depth, projective dimension and Castelnuovo-Mumford regularity depends heavily on applying the depth lemma (or similar results for the other two invariants), it is not a trivial work for establishing parallel results for Stanley depth and Stanley regularity. For instance, it is still conjectured \cite{arXiv:1011.6462} (but not established) that
\[
\sdepth(I)\ge \sdepth(S/I)+1.
\]

\section{Preliminaries}
We begin by recalling basic notation and terminology from commutative algebra and combinatorics. For further reading, one can refer to \cite{MR1251956}, \cite{herzog2013survey}, \cite{MR2724673} and \cite{arXiv:1310.7912}. 

\subsection{Clutters}
A \Index{clutter} $\calC=(V,E)$ over the vertex set $V(\calC)=V$ consists of a collection $E(\calC)=E$ of subsets of $V$, called the \Index{edges} of $\calC$, with the property that no edge contains another. Clutters are also known as \Index{simple hypergraphs} or \Index{Sperner systems}. 
We will only consider clutters whose vertex set is finite.

Two distinct vertices in $V(\calC)$ are \Index{neighbors} if there is an edge of $\calC$ that contains these vertices.
A vertex $v\in V(\calC)$ is \Index{isolated} if it does not appear in any edge in $E(\calC)$.  We will write $\is(\calC)$ for the set of isolated vertices and $\calC^\Red$ for the clutter from $\calC$ with its isolated vertices removed. 

An edge $e\in E(\calC)$ is \Index{trivial} if it contains only one vertex in $V(\calC)$. Trivial edges are also called \Index{isolated loops}. When the cardinality of each edge equals a fixed integer $d\ge 2$, the clutter $\calC$ is \Index{$d$-uniform}.

A collection of edges in $\calC$ is called a \Index{matching} if the edges in this collection are pairwise disjoint. The maximum size of a matching in $\calC$ is called its \Index{matching number}. The minimal size of a maximal matching is called the \Index{minimax matching number}.

For a nonempty subset $A$ of vertices in $\calC$, let $\calC+A$ denote the clutter whose edges are the minimal sets of $E(\calC)\cup \Set{A}$ and whose vertex set is still $V(\calC)$. Meanwhile, let $\calC:A$ be the clutter whose edges are the minimal sets of $\Set{e\setminus A: e\in E(\calC)}$ and whose vertex set is $V(\calC)\setminus A$.

For simplicity, we often identify vertex sets with subsets of the variables $\Set{\bdx}:=\Set{x_1,\dots,x_n}$. If $A$ is a subset of $\Set{\bdx}$, we write $\bdx^A$ for the squarefree monomial $\prod_{x\in A} x$ in $S=\KK[x_1,\dots,x_n]$. Thus, the clutter $\calC$ corresponds to a squarefree monomial ideal
\[
I(\calC)=\braket{\bdx^e:e\in E(\calC)} \subset S.
\]
This ideal is called the \Index{edge ideal} of $\calC$. Naturally, the clutters $\calC+A$ and $\calC:A$ correspond to the squarefree monomial ideals $\braket{I(\calC),\bdx^A}$ and $I(\calC):_S \bdx^A$ respectively.

\subsection{Filtration depth and Stanley depth}

Recall that a sequence
\[
\calF: 0=M_0\subset M_1 \subset \cdots \subset M_m=M
\]
of $\ZZ^n$-graded submodules of $M$ is a \Index{prime filtration} if each $M_i/M_{i-1}\isom (S/P_i)(-\bda_i)$ for some integral vectors $\bda_i\in \ZZ^n$ and some monomial prime ideals $P_i$. The set of the primes $\Set{P_1,\dots,P_m}$ is the \Index{support} of $\calF$, which shall be denoted by $\supp(\calF)$. Now
\[
\fdepth(\calF):=\min\Set{\dim(S/P):P\in \supp(\calF)}
\]
is the \Index{filtration depth} of $\calF$ and
\[
\fdepth(M):=\max\Set{\fdepth(\calF): \text{ $\calF$ is a prime filtration of $M$}}
\]
is the \Index{filtration depth} of $M$.

On the other hand, if $M$ is a nonzero finitely generated $\ZZ^n$-graded $S$-module, $u\in M$ is a homogeneous element and $Z$ is a subset of $\Set{\bdx}$, then $u\KK[Z]$ is the $\KK$-subspace of $M$ generated by all elements $uv$ where $v$ is a monomial in $\KK[Z]$. A presentation of $M$ as a finite direct sum of such spaces $\calD:M=\bigoplus_{i=1}^r u_i\KK[Z_i]$ is called a \Index{Stanley decomposition} of $M$. Set $\sdepth(\calD)=\min\mySet{|Z_i|:i=1,\dots,r}$ and
\[
\sdepth(M)=\max\Set{\sdepth(\calD) : \text{$\calD$ is a Stanley decomposition of $M$} }.
\]

We have the following relations among depth, filtration depth and Stanley depth.

\begin{lemma}
  [{\cite[Proposition 1.3]{arxiv.0712.2308}}]
  \label{fdepthsdepthdepth}
  Let $M$ be a nonzero finitely generated $\ZZ^n$-graded $S$-module. Then
  \[
  \fdepth(M)\le \depth(M),\ \sdepth(M) \le \min\Set{\dim(S/P):P\in \Ass(M)}.
  \]
\end{lemma}

Later in this paper, we need the following facts for Stanley depth.

\begin{lemma}
  \label{basic-facts}
  Let $M$ be a nonzero finitely generated $\ZZ^n$-graded $S$-module. 
  \begin{enumerate}[a]
    \item The module $M$ is $S$-free if and only if $\sdepth(M)=n$. 
    \item If $\sdepth(M)=0$, then $\depth(M)=0$.
    \item If $\depth(M)=0$ and $\dim_{\KK}M_\bda \le 1$ for all $\bda\in\ZZ^n$, then $\sdepth(M)=0$.
  \end{enumerate}
\end{lemma}

The proof can be found, for instance, in \cite[Lemma 1.2, Theorem 1.4]{arXiv.org:0808.2657} and \cite[Proposition 2.13]{bruns-2009}.

\subsection{Stanley regularity}
The starting point is Terai's duality theorem via Alexander dual for squarefree monomial ideals. Let $I=(x_{11}\cdots x_{1i_1},\dots,x_{s1}\cdots x_{si_s})\subset S=\KK[x_1,\dots,x_n]$ be a squarefree monomial ideal. Then the \Index{Alexander dual} of $I$ is 
\[
I^\vee=(x_{11},\dots,x_{1i_1})\cap \cdots \cap (x_{s1},\dots,x_{si_s})
\]
with the property that $(I^\vee)^\vee=I$. Terai \cite[Corollary 0.3]{MR1715588} (see also \cite[8.1.10]{MR2724673}) proved that
\[
\projdim(I)=\reg(S/I^\vee).
\]

As established in \cite{MR1833633} and \cite{MR1741555}, Alexander duality can also be extended to finitely generated squarefree modules. Let $M$ be such an module and 
$\calD:M=\bigoplus_{i=1}^r u_i\KK[Z_i]$ be a squarefree Stanley decomposition, then the \Index{Stanley regularity} of $\calD$ is
\[
\sreg(\calD):=\max\Set{\deg(u_i):1\le i\le m}
\]
and
the \Index{Stanley regularity} of $M$ is
\[
\sreg(M):=\min\Set{\sreg(\calD):\text{$\calD$ is a Stanley decomposition of $M$}}.
\]

Similar to \eqref{middle-sdepth}, it is straight forward to see that for a short exact sequence $0\to M\to N \to L \to 0$ of finitely generated squarefree modules, we have
\[
\sreg(N) \le \min\Set{\sreg(M),\ \sreg(L)}.
\]

The following key result plays the same role as the Terai's duality theorem.

\begin{lemma}
  [{\cite[Theorem 3.7]{MR2563149}, \cite[Corollary 46]{herzog2013survey}}]
  \label{sdepth-sreg}
  If $I$ is a squarefree monomial ideal of $S$, then $\sreg(S/I)=n-\sdepth(I^\vee)$ and $\sreg(I)=n-\sdepth(S/I^\vee)$.
\end{lemma}

\begin{remark}
  \label{reduction}
  Suppose $I$ is a squarefree monomial ideal of $S=\KK[x_1,\dots,x_n]$ and $S'=S[x_{n+1}]$. Then $\sreg(IS')=\sreg(I)$ and $\sreg(S'/(IS'))=\sreg(S/I)$ by virtue of the above lemma together with \cite[Lemma 3.6]{arxiv.0712.2308}. In other words, the isolated vertices are irrelevant for computing the Stanley regularity of edge ideals.
\end{remark}

The following inequality is dual to the Stanley's conjecture \eqref{StanleyConjecture}.

\begin{conjecture}
  [{\cite{MR2563149}}]
  \label{sregConjecture}
  Let $J\subset I$ be squarefree monomial ideals. Then $\sreg(I/J)\le \reg(I/J)$.
\end{conjecture}

\begin{remark}
  Let $I\subset S$ be a squarefree monomial ideal. Then $\sreg(S/I)=0$ if and only if $I$ is a prime ideal generated by a set of variables. To see this, it suffices to mention that $\sdepth(I^\vee)=n$ if and only if $I^\vee$ is principal.
\end{remark}

\section{Edge domination and Stanley depth}
Let $\calC$ be a clutter. The following definition is due to \cite{arXiv:1301.2665}.

\begin{definition}
  \label{ied}
  A collection $F$ of edges in $\calC$ is called \Index{edgewise dominant} if for every vertex $v\in V(\calC^\Red)$ which is not contained in some edge of $F$ or contained in a trivial edge, it has a neighbor contained in some edge of $F$.  The \Index{index of edgewise domination} is the number
  \[
  \epsilon(\calC)=\min\mySet{|F|:F\subset E(\calC)\text{ is edgewise dominant}}.
  \]
\end{definition}

Dao and Schweig \cite[Theorem 3.2]{arXiv:1301.2665} proved that $\projdim(S/I(\calC))\le |V(\calC^\Red)|-\epsilon(\calC^\Red)$. This result, by a theorem of Auslander and Buchsbaum \cite[Theorem 1.3.3]{MR1251956}, is equivalent to saying that
\begin{equation}
  \depth(S/I(\calC)) \ge \epsilon(\calC)+n-|V(\calC^\Red)|.
  \label{DS}
\end{equation}
It is clear that Stanley's conjecture \eqref{StanleyConjecture} for $M=S/I(\calC)$ with the above inequality implies that
\begin{equation}
  \sdepth(S/I(\calC)) \ge \epsilon(\calC)+n-|V(\calC^\Red)|.
  \label{DS-sreg}
\end{equation}
This is the result that we want to establish in this section. To this end, let us go over some basic concepts and constructions from the original paper \cite{arXiv:1301.2665}.

\begin{definition}
  A collection $\Phi$ of clutters is \Index{hereditary} if for any clutter $\calC\in \Phi$ and any nonempty subset $A$ of vertices of $\calC$, the clutters $\calC:A$, $\calC+A$  and $\calC^\Red$ are all in $\Phi$.
\end{definition}

Suppose $\Phi$ is a hereditary collection of clutters and $f:\Phi\to \NN$ is a function. We consider the following conditions for $f$.
\begin{enumerate}[label=(\textbf{DS}.\arabic*)]
  \item \label{DS0} $f(\calC)\le \max\Set{f(\calC+A),\ f(\calC:A)}$ for all $\calC\in \Phi$ and nonempty $A\subset V(\calC)$.
  \item \label{DS1} $f(\calC^\Red)=f(\calC)$ for all $\calC\in \Phi$.
  \item \label{DS2} $f(\calC)\le |V(\calC)|$ when $E(\calC)=\emptyset$.
  \item \label{DS2'} $f(\calC)=0$ when $E(\calC)=\emptyset$.
  \item \label{DS3} $f(\calC)=0$ when $\calC$ has only trivial edges.
  \item \label{DS3'} $f(\calC)\le |V(\calC)|$ when $\calC$ has only trivial edges.
  \item \label{DS4} For any $\calC\in \Phi$ with at least one non-trivial edge, there exists a sequence of nonempty subsets $A_1,\dots,A_t$ of $V(\calC)$ such that for the clutters $\calC_i:=\calC+\sum_{j=1}^i A_j$, the following properties are satisfied:
    \begin{itemize}
      \item $|\is(\calC_t)|>0$ and $f(\calC_t)+|\is(\calC_t)|\ge f(\calC)$, and
      \item for each $i$, $f(\calC_{i-1}:A_i)+|\is(\calC_{i-1}:A_i)|+|A_i|\ge f(\calC)$.
    \end{itemize}
\end{enumerate}

\begin{observation}
  When $g(\calC)=\projdim(\KK[x_1,\dots,x_n]/I(\calC))$ for clutters $\calC$ with $V(\calC)=\Set{x_1,\dots,x_n}$, it is clear that $g$ satisfies the above conditions \ref{DS0}, \ref{DS1}, \ref{DS2'} and \ref{DS3'}.
\end{observation}

The following key lemma with its proof is adapted from \cite[Lemma 3.3]{arXiv:1301.2665}.

\begin{lemma}
  Let $\Phi$ be a hereditary class of clutters and $f,g:\Phi\to \NN$ two functions such that
  \begin{enumerate}[a]
    \item $f$ satisfies conditions \ref{DS1}, \ref{DS2}, \ref{DS3} and \ref{DS4}, and
    \item $g$ satisfies conditions \ref{DS0}, \ref{DS1}, \ref{DS2'} and \ref{DS3'}.
  \end{enumerate}
  Then for any $\calC\in \Phi$, $f(\calC)+ g(\calC)\le |V(\calC^\Red)|$.
  \label{key-lemma}
\end{lemma}

\begin{proof}
  We prove by induction on $|V(\calC)|$.  Because of the condition \ref{DS1}, we may assume that $\calC=\calC^\Red$ is a clutter without any isolated vertex. Meanwhile, by the conditions \ref{DS2}, \ref{DS2'}, \ref{DS3} and \ref{DS3'}, we may further assume that $\calC$ has at least one non-trivial edge. Therefore, there exists a sequence of sets $A_1,\dots,A_t$ as in the condition \ref{DS4} for the function $f$.

  By the condition \ref{DS0} for the function $g$, we are reduced to the following two cases.
  \begin{enumerate}[a]
    \item $g(\calC)\le g(\calC_1)\le \cdots \le g(\calC_t)=g({\calC_t}^\Red)$. Since $|\is(\calC_t)|>0$, by induction hypothesis, we have
      \begin{align*}
        g(\calC) &\le  g({\calC_t}^\Red)\le |V({\calC_t}^\Red)|-f(\calC_t) \\
        &\le  (|V(\calC)|-|\is(\calC_t)|)-(f(\calC)-|\is(\calC_t)|)\\
        &= |V(\calC)|-f(\calC).
      \end{align*}
    \item $g(\calC)\le g(\calC_1)\le \cdots g(\calC_{i-1})\le g(\calC_{i-1}:A_i)$ for some integer $i\in \Set{1,2,\dots,t}$. Since
      \[
      |V(\calC_{i-1}:A_i)|=|V(\calC_{i-1})|-|A_i|<|V(\calC_{i-1})|=|V(\calC)|,
      \]
      we can apply the induction hypothesis and get
      \begin{align*}
        g(\calC) &\le  g(\calC_{i-1}:A) \le |V((\calC_{i-1}:A_i)^\Red)|-f(\calC_{i-1}:A_i) \\
        &\le  (|V({\calC_{i-1}:A_i})|-|\is(\calC_{i-1}:A_i)|)-(f(\calC)-|\is(\calC_{i-1}:A_i)|-|A_i|)\\
        &=  (|V(\calC_{i-1})|-|A_i|)-f(\calC)+|A_i| =  |V(\calC)|-f(\calC). \qedhere
      \end{align*}
  \end{enumerate}
\end{proof}

\begin{lemma}
  Let $\Phi$ be the collection of all clutters whose vertex set is a subset of $\Set{x_1,\dots,x_n}$. Then the index of edgewise dominant $\epsilon$ is a function that satisfies the conditions \ref{DS1}, \ref{DS2}, \ref{DS3} and \ref{DS4}.
  \label{verification}
\end{lemma}

\begin{proof}
  It follows easily from the definition that $\epsilon$ satisfies conditions \ref{DS1}, \ref{DS2} and \ref{DS3}. As for the condition \ref{DS4}, let $\calC$ be a clutter with at least one non-trivial edge. Let $x$ be a vertex in such an edge and $y_1,\dots,y_t$ be the neighbors of $x$. If we take $A_i=\Set{y_i}$, then the proof of \cite[Theorem 3.2]{arXiv:1301.2665} shows that $\epsilon$ satisfies the condition \ref{DS4}, which we will not repeat here.
\end{proof}

Here is the main result of this section.

\begin{theorem}
  \label{main-theorem}
  Let $\calC$ be a clutter and $I(\calC)$ the corresponding edge ideal in $S$. Then
  \[
  \fdepth(S/I(\calC))\ge \epsilon(\calC)+n-|V(\calC^\Red)|.
  \]
\end{theorem}

\begin{proof}
  Let $\Phi$ be the collection of all clutters whose vertex set is a subset of $\Set{x_1,\dots,x_n}$.  It suffices to show that $g(\calC):=n-\fdepth(S/I(\calC))$ is a function from $\Phi$ to $\NN$ that satisfies the conditions \ref{DS0}, \ref{DS1}, \ref{DS2'} and \ref{DS3'}.
  \begin{enumerate}[a]
    \item Since $\fdepth(S/I(\calC))$ takes value in $\Set{0,1,\dots,n=\dim(S)}$, $g$ is a function from $\Phi$ to $\NN$.
    \item Let $\calC$ be a clutter in $\Phi$ and $A\subset V(\calC)$. There is a natural short exact sequence
      \[
      0\to {S}/(I(\calC):\bdx^A)\to {S}/{I(\calC)} \to {S}/{(I(\calC),\bdx^A)} \to 0.
      \]
      Obviously, a prime filtration filtration of ${S}/(I(\calC):\bdx^A)$ can be combined with a prime filtration of ${S}/{(I(\calC),\bdx^A)}$ to yield a prime filtration of ${S}/{I(\calC)}$. Hence
      \[
      \fdepth(S/I(\calC)) \ge \min\Set{\fdepth(S/(I(\calC):\bdx^A),\ \fdepth(S/(I(\calC),\bdx^A)))},
      \]
      and $g$ satisfies the condition \ref{DS0}.

    \item Since $I(\calC)=I(\calC^\Red)$ in $S$, the function $g$ satisfies the condition \ref{DS1}.
    \item If $\calC\in \Phi$ with $E(\calC)=\emptyset$, then $I(\calC)=0$. Hence $g(\calC)=n-\fdepth(S)=0$ and satisfies the condition \ref{DS2'}.
    \item If $\calC\in \Phi$ has only trivial edges, $I(\calC)=\braket{y:y\in V(\calC)}$.  Thus $\fdepth(S/I(\calC))=n-|V(\calC)|$ and $g$ satisfies the condition \ref{DS3'}. \qedhere
  \end{enumerate}
\end{proof}

Notice that the above result is slightly stronger than \cite[Theorem 3.2]{arXiv:1301.2665}:

\begin{theorem}
  \label{main-ied}
  Let $\calC$ be a clutter and $I(\calC)$ the corresponding edge ideal in $S$. Then
  \[
  \min\Set{\depth(S/I(\calC)),\ \sdepth(S/I(\calC))}\ge \epsilon(\calC)+n-|V(\calC^\Red)|.
  \]
\end{theorem}

\begin{proof}
  Since $\fdepth(M)\le \min\Set{\depth(M),\sdepth(M)}$ by Lemma \ref{fdepthsdepthdepth}, this result follows from Theorem \ref{main-theorem} immediately.
\end{proof}

In particular, we get the nice lower bound \eqref{DS-sreg} for the Stanley depth of $S/I(\calC)$, as expected.

\begin{cor}
  Let $\calC$ be a clutter and $I(\calC)$ the corresponding edge ideal in $S$. Then there is a prime filtration of $S/I(\calC)$
  \[
  \calF: 0=M_0\subset M_1 \subset \cdots \subset M_m=M
  \]
  such that for each $i$, $M_i/M_{i-1}\isom (S/P_i)(-\bda_i)$ is of dimension at least $\epsilon(\calC)+n-|V(\calC^\Red)|$ and $\bda_i$ is a squarefree vector in $\NN^n$.
\end{cor}

\begin{proof}
  As pointed out in \cite[Corollary 2.5]{arxiv.0712.2308}, the filtration depth of $S/I(\calC)$ can be computed by checking special partitions of the poset $P_{S/I(\calC)}^{\bm{1}}$ whose elements are the squarefree monomials in $S\setminus I(\calC)$. Since $\fdepth(S/I(\calC))\ge\epsilon(\calC)+n-|V(\calC^\Red)|$ by Theorem \ref{main-theorem}, the expected filtration exists by virtue of \cite[Theorem 2.4(a)]{arxiv.0712.2308}.
\end{proof}

\section{Splitting and Stanley regularity}
If $I$ is a squarefree monomial ideal minimally generated by monomials $u_1,\dots,u_m$ and $w$ is the smallest number $t$ such that there exists integers $1\le i_1 < i_2 < \cdots < i_t \le m$ such that
\[
\lcm(u_{i_1},u_{i_2},\dots,u_{i_t})=\lcm(u_1,u_2,\dots,u_m),
\]
then the number $\deg \lcm(u_1,\dots,u_m)-w$ is called the \Index{cosize} of $I$, denoted by $\cosize(I)$. Now, dual to the inequality \eqref{HPV}, we have
\begin{equation}
 \sreg(S/I(\calC))\le \cosize(I(\calC))
  \label{cosize}
\end{equation}
by \cite[Corollary 3.4]{arXiv:1011.6462}. 

Suppose in the above setting each monomial $u_i$ contains a \Index{free variable}, i.e., a variable that divides this $u_i$ but not any other monomial generator. Then the Taylor resolution of $S/I$ is minimal by \cite[Proposition 4.1]{arXiv:1211.4301} and $\reg(S/I)$ is exactly the number $|X|-m$ where $X$ is the set of variables showing in these $u_i$'s. Due to Conjecture \ref{sregConjecture}, it is natural to ask whether the inequality
  \begin{equation}
    \label{LM-type}
  \sreg(S/I)\le |X|-m
\end{equation}
holds in general. As a matter of fact,
the inequality \eqref{LM-type} holds as a special case of the inequality \eqref{cosize} by recognizing that $\cosize(I)=|X|-m$. On the other hand, equality does not hold for \eqref{LM-type} in general. For instance, when $I=\braket{x_1x_2\cdots x_n}$ is a principal ideal, then $\sreg(S/I)=\floor{\frac{n}{2}}<n-1$ for $n\ge 3$.

Vertices of a clutter $\calC$ that correspond to the free variables for its edges ideal $I(\calC)$ are also free, i.e., an edge $e\in E(\calC)$ is said to contain a \Index{free vertex} if there exists some vertex $x\in e$ such that $x$ does not belong to any other edges in $\calC$.
The subsequent generalization of the inequality \eqref{LM-type} is parallel to \cite[Theorem 4.9]{arXiv:1211.4301}. We adopt the following version, rephrased by \cite[Theorem 4.20]{arXiv:1310.7912}.

\begin{theorem}
  \label{LM}
  Let $\calC=(V,E)$ be a clutter and $\calC'=(V,E')$ be the clutter obtained by removing all edges with free vertices from $\calC$. Let $\beta(\calC')$ be the matching number of $\calC'$. Then
  \[
  \sreg(S/I(\calC))\le |V|-|E|+|E'|-\beta(\calC').
  \]
\end{theorem}

Since our proof for Theorem \ref{LM} and Lemma \ref{2collage} depends heavily on the splitting method in \cite{arXiv:1011.6462}, we will outline the key ingredients here. 

\begin{construction}
  \label{splitting}
  Let $I=\bigcap_{i=1}^s P_i$ be an irredundant primary decomposition of the squarefree monomial ideal $I$ in $S$. All the $P_i$'s are necessarily generated by subsets of $\Set{x_1,\dots,x_n}$. We will take the variables in some specific $P_i$ as a splitting set. Without loss of generality, we may choose $P_1$ and assume that $P_1=\braket{x_1,x_2,\dots,x_r}$. Write $S'=\KK[x_1,\dots,x_r]$ and $S''=\KK[x_{r+1},\dots,x_n]$. For each $P_i$, let $P_i'=P_i\cap S'$ and $P_i''=P_i\cap S''$.  Now, for each subset $\tau \subset [s]:=\Set{1,2,\dots,s}$, let $I_\tau$ be the $\ZZ^n$-graded $\KK$-vector space spanned by the set of monomials of the form $w=uv$, where $u\in S'$ and $v\in S''$ are monomials with $u\in \bigcap_{j\notin \tau}P_j \setminus \sum_{j\in\tau} P_j$ and $v\in \bigcap_{j\in\tau}P_j$.  Thus, by \cite[Proposition 2.1]{arXiv:1011.6462}, $I=\bigoplus_{\tau \subset [s]} I_\tau$ is a decomposition of $I$ as a direct sum of $\ZZ^n$-graded $\KK$-subspaces of $I$ with $I_{[s]}=0$.
  By the explanation after \cite[Proposition 2.1]{arXiv:1011.6462}, we can write $I_\tau$ as $I_\tau=H_\tau\otimes_\KK L_\tau$ where
  \begin{equation}
    H_\tau=\frac{\bigcap_{j\notin \tau}P_j'+\sum_{j\in\tau}P_j'}{\sum_{j\in\tau}P_j'}
    \label{Htau}
  \end{equation}
  and $L_\tau=\bigcap_{j\in \tau}P_j''$.
\end{construction}

\begin{proof}[{Proof of Theorem \ref{LM}}]
  Without loss of generality, we may assume that edges $e_1,\dots,e_a$ are removed from $\calC$ to get the clutter $\calC'$. We also assume that $\beta(\calC')=b$ such that edges $e_{a+1},\dots,e_{a+b}$ form a maximal matching in $\calC'$. Let the remaining edges be $e_{a+b+1},\dots,e_{a+b+c}$. Since we can also assume that $\calC$ contains no isolated vertices, we are reduced to prove that
  \[
  \sreg(S/I(\calC)) \le n-(a+b+c)+(b+c)-b=n-a-b.
  \]

  For each edge $e_i$ of $\calC$, there is a corresponding monomial prime ideal $P_i=\braket{x_j:x_j\in e_i}\subset S$. Now $I(\calC)^\vee =\bigcap_{i=1}^s P_i$. Since $\calC$ contains no isolated vertices, $\sum_{i=1}^s P_i=\braket{x_1,\dots,x_n}$.
   We need to prove that
  \begin{equation}
    \sdepth\left(\bigcap_{i=1}^{a+b+c}P_i\right)\ge a+b.
    \label{eq10}
  \end{equation}

  We will prove by induction on the number $a+b+c$. When $a+b+c=1$, this is trivial. When $b=0$, $\calC'$ has no edge and all the edges of $\calC$ contain free vertices. In this situation, $\size(I(\calC)^\vee)=a-1$. Thus we are done, thanks to the inequality \eqref{HPV}.

  In the following, we consider the case when $b\ge 1$ and assume that \eqref{eq10} holds for smaller $a+b+c$. 
  We will split using the variables in $P_{a+1}$ as in Construction \ref{splitting} and define the rings $S'$ and $S''$ accordingly. For each $k=1,\dots,a$, we may assume that $x_{i_k}$ is a free vertex in $E_k$. Necessarily $x_{i_k}\in P_k''$. Another key observation is that $P_i'=0$ for $i=a+2,\dots,a+b$.
  
  The dual ideal $I(\calC)^\vee$ has a $\ZZ^n$-graded $\KK$-subspace decomposition $I(\calC)^\vee=\bigoplus_{\tau\subset [a+b+c]}I_\tau$ with $I_{[a+b+c]}=0$.  It suffices to consider the case when $I_\tau=H_\tau\otimes_\KK L_\tau\ne 0$; whence $H_\tau\ne 0$ and $L_\tau\ne 0$.  But for $H_\tau\ne 0$ in the presentation \eqref{Htau}, we need $a+1\notin \tau$ and $i\in \tau$ for $i=a+2,\dots,a+b$.  Notice that $H_\tau$ is isomorphic to a squarefree monomial ideal in $\KK[x_i\mid x_i\in P_1\setminus \sum_{j\in\tau} P_j ]$. Thus, when $H_\tau$ is nonzero, $\sdepth_{S'}(H_\tau)\ge 1$ by Lemma \ref{basic-facts}.
  
  Next, we demonstrate that $\sdepth(L_\tau)\ge a+b-1$ when $L_\tau\ne 0$.  
  Notice that $L_\tau=\bigcap_{j\in \tau}P_j''$ with $\Set{a+2,\dots,a+b}\subset \tau$. Suppose that the intersection $L_\tau=\bigcap_{j\in\tau'}P_j''$ is an irredundant primary decomposition of $L_\tau$. Obviously $\tau'$ is a non-empty subset of $\tau$. For each $i\in \Set{a+2,\dots,a+b}$, if $i\in \tau\setminus\tau'$, there must exists some $P_{k_i}''$ with $P_{k_i}''\subset P_{i}''$ and $k_i\in \tau'$. Since $P_1'',\dots,P_a''$ all contain free variables while $e_{a+1},\dots,a_{a+b}$ are pairwise disjoint, this $k_i\in \tau\cap\Set{a+b+1,\dots,a+b+c}$ and $P_{k_i}''$ is contained in at most one such $P_j''$ when we limit $j$ to $\Set{a+2,\dots,a+b}$. 
  
  Let us check the clutter $\widehat{\calC}$ containing edges corresponding to the prime ideals $P_i''$ with $i\in \tau'$. For each $i\in \Set{1,\dots,a}\cap \tau'$, let $\hat{e}_i$ be the corresponding edge. This $\hat{e}_i$ still contains a free vertex. For each $i\in \Set{a+2,\dots,a+b}$, if $i\in\tau'$, let $\hat{e}_i$ be the edge corresponding to $P_i''$; otherwise, let $\hat{e}_i$ be the edge corresponding to $P_{k_i}''$. The edges $\hat{e}_{a+2},\dots,\hat{e}_{a+b}$ are pairwise disjoint. Thus, by induction hypothesis, we have 
  \[
  \sdepth_{S''}(L_\tau)\ge (n-\Ht(P_{a+1}))-\Ht(\sum_{j\in\tau'}P_j'')+|\tau'\cap\Set{1,\dots,a}|+(b-1).
  \]
  When $k\in\Set{1,\dots,a}\setminus \tau'$, $x_{i_k}\in S''\setminus\sum_{j\in\tau'}P_j''$. Hence 
  \[
  n-\Ht(P_{a+1})-\Ht(\sum_{j\in\tau}P_j'')\ge |\Set{1,\dots,a}\setminus \tau'|.
  \]
  Consequently, 
  \[
  \sdepth(L_\tau)\ge |\Set{1,\dots,a}\setminus \tau'|+|\tau'\cap\Set{1,\dots,a}|+(b-1)=a+b-1.  
  \]
  Now
  \[
  \sdepth(I_\tau)\ge \sdepth(H_\tau)+\sdepth(L_\tau)\ge 1+a+b-1 = a+b,
  \]
  as expected. Finally, we arrive at the desired inequality
  \[
  \sdepth_S(I^\vee)\ge \min\Set{\sdepth_S(I_\tau):\tau\subsetneq [r+s] \text{ and } I_\tau\ne 0}\ge a+b. \qedhere
  \]
\end{proof}

\section{Packing and Stanley regularity}
Let $I_1, \dots, I_s$ be squarefree monomial ideals in $S$. In \cite{MR2259083}, Kalai and Meshulam obtained  the following results:
\begin{enumerate}[a]
  \item $\reg(S/\sum_{i=1}^s I_i) \le \sum_{i=1}^s \reg(S/I_i)$, and
  \item $\reg(\bigcap_{i=1}^s I_i) \le \sum_{i=1}^s \reg(I_i)$.
\end{enumerate}
These results were later extended to arbitrary (not necessarily squarefree) monomial ideals by Herzog \cite{MR2317642}. Since the above inequalities play an indispensable role in the research of \cite{arXiv:1009.2756} and \cite{arXiv:1301.6779}, we will start by generalizing these results to the Stanley regularity of squarefree monomial ideals.

\begin{lemma}
  \label{duality}
   Let $I_1, \dots, I_s$ be squarefree monomial ideals in $S=\KK[x_1,\dots,x_n]$. Then
   \[
   \left(\sum_{i=1}^s I_i\right)^\vee = \bigcap_{i=1}^s I_i^\vee \quad \text{and} \quad \left(\bigcap_{i=1}^s I_i\right)^\vee = \sum_{i=1}^s I_{i}^\vee.
   \]
\end{lemma}

\begin{proof}
  The first equality follows from definition. The second equality follows from the first one by using the duality $(I^\vee)^\vee=I$.
\end{proof}

\begin{lemma}
  \label{KM-type}
  Let $I_1, \dots, I_s$ be squarefree monomial ideals in $S=\KK[x_1,\dots,x_n]$. Then
\begin{enumerate}[a]
  \item $\sreg(S/\sum_{i=1}^s I_i) \le \sum_{i=1}^s \sreg(S/I_i)$, and
  \item $\sreg(\bigcap_{i=1}^s I_i) \le \sum_{i=1}^s \sreg(I_i)$.
\end{enumerate}
\end{lemma}

\begin{proof} 
  \begin{enumerate}[a]
    \item We have
      \begin{align*}
        \sreg\left(S\Big/\sum_{i=1}^s I_i\right) = & n-\sdepth\left(\sum_{i=1}^s I_i\right)^\vee = n-\sdepth\left(\bigcap_{i=1}^s I_i^\vee\right) \\
        \stackrel{(\ast)}{\le}& n-\left(\sum_{i=1}^s \sdepth(I_i^\vee)-(s-1)n\right) \\
        = & \sum_{i=1}^s (n-\sdepth(I_i^\vee)) = \sum_{i=1}^s \sreg(S/I_i).
      \end{align*}
      For the inequality $(\ast)$ above, we have applied \cite[Corollary 2.11(1)]{arXiv:1107.3359}. 
    \item  This result can be similarly proved by applying  \cite[Corollary 2.11(4)]{arXiv:1107.3359}. \qedhere
  \end{enumerate}
\end{proof}

\subsection{Simple graphs}

In this subsection, we will restrict ourselves to the simple graphs, namely those clutters whose edges all contain exactly two distinct vertices. We will in general denote such a simple graph by $G$ instead of $\calC$. And $\overline{G}$ shall be the complement graph of $G$.

\begin{corollary}
  \label{mk2}
  If $G=mK_2$ is the simple graph of $m$ disjoint edges, then $\sreg(S/I(G))=\ceil{\frac{m}{2}}$.
\end{corollary}

\begin{proof}
  It is straightforward to see that the claim holds when $m=1$ and $m=2$. Thus, after partitioning $mK_2=2K_2+\cdots+2K_2$ when $m$ is even and $mK_2=2K_2+\cdots+2K_2+K_1$ when $m$ is odd, we get $\sreg(S/I(mK_2))\le \ceil{\frac{m}{2}}$ by Lemma \ref{KM-type}(a). 

  On the other hand, we may assume that $G$ contains no isolated vertex. Now, all squarefree monomials in $I(mK_2)^\vee$ have degree at least $m$. It is easy to see that there are exactly $2^m$ of them having degree $m$ and $m2^{m-1}$ of them having degree $m+1$. Thus $\sdepth(I(mK_2)^\vee)\le m+\floor{\frac{m2^{m-1}}{2^m}}=m+\floor{\frac{m}{2}}$ by \cite[Lemma 2.4]{arXiv:1104.2412}. Thus, $\sreg(S/I(mK_2))\ge 2m-(m+\floor{\frac{m}{2}})=\ceil{\frac{m}{2}}$.
\end{proof}

If $A\subset V(G)$, then $G\setminus A$ denotes the induced subgraph on $V(G)\setminus A$. When $A=\Set{x_v}$ consists of exactly one vertex, we will write $G\setminus x_v$ instead of $G\setminus A$.

A \Index{clique} of $G$ is a subset of pairwise adjacent vertices. Cliques are not required to be maximal.
Now, for a vertex $x\in V(G)$, the set of \Index{neighbours} of $x$ is given by
\[
N(x)=\Set{y\in V(G)| \Set{x,y}\in E(G)}.
\]
The vertex $x$ is \Index{simplicial} if $N(x)$ induces a clique in $G$. 

\begin{lemma}
  \label{simplicial-vertex}
  Let $G$ be a simple graph with $V(G)=\Set{x_1,\dots,x_n}$ and $x_v$ be a simplicial vertex of $G$. Let $G_1=G\setminus x_v$ and $I=I(\overline{G})$, $J=I(\overline{G_1})$ be the corresponding edge ideals in $S=\KK[x_1,\dots,x_n]$. 
  \begin{enumerate}[a]
    \item  If $J\ne 0$, then $\sreg(S/I)\le \sreg(S/J)$.
    \item If $J=0$, then $\sreg(S/I)=1$.
  \end{enumerate}
\end{lemma}

\begin{proof}
  \begin{enumerate}[a]
    \item Suppose $J\ne 0$. We will follow the strategy of \cite[Theorem 2.7]{MR2943752}.  By Lemma \ref{sdepth-sreg}, it suffices to show that $\sdepth(I^\vee)\ge \sdepth(J^\vee)$. We may assume that $N(x_v)=\Set{x_1,\dots,x_{v-1}}$ and the minimal monomial generators of $J$ belong to $\KK[x_1,\dots,\widehat{x_v},\dots,x_n]$. Now $I=J+\braket{x_vx_i:v<i\le n}$. Since $x_v$ is a simplicial vertex of $G$, we have $x_{v+1}\cdots x_n\in J^\vee$. Now,
      \begin{align*}
        I^\vee & = J^\vee \cap \left(\bigcap_{i=v+1}^n \braket{x_v,x_i}\right) \\
        & = J^\vee \cap (x_v,x_{v+1}\cdots x_n) \\
        & = ( \braket{x_v}\cap J^\vee) + \braket{x_{v+1}\cdots x_n}
      \end{align*}
      with the property that $( \braket{x_v}\cap J^\vee) \cap \braket{x_{v+1}\cdots x_n}=\braket{x_vx_{v+1}\cdots x_n}$ and $\braket{x_v}\cap J^\vee=x_vJ^\vee$. Therefore, by \cite[Proposition 2.6]{bruns-2009}, we have
      \begin{align*}
        \sdepth(I^\vee)&\ge \min\Set{\sdepth(x_vJ^\vee),\ \sdepth(\braket{x_{v+1}\cdots x_n}/\braket{x_v\cdots x_n})} \\
        & = \min\Set{\sdepth(J^\vee),\ \sdepth(S/\braket{x_v})} \\
        & = \min\Set{\sdepth(J^\vee),\ n-1} \\
        & = \sdepth(J^\vee).
      \end{align*}
      Notice that $\sdepth(J^\vee)=n$ if and only if $J^\vee$ is principal, whence $J$ is a prime ideal generated by a set of variables. But this cannot happen for the edge ideal of a finite simple graph, unless $J=0$.
    \item Suppose $J=0$. We might assume that $x_v=x_1$ and $I=x_1\braket{x_2,\dots,x_n}$. Now $I^\vee=\braket{x_1,x_2\cdots x_n}$ is two-generated, thus $\sdepth(I^\vee)=n-1$ by \cite[Corollary 3.5]{arxiv.0712.2308}. It follows that $\sreg(I)=1$. \qedhere
  \end{enumerate}
\end{proof}

A graph $G$ is \Index{chordal} if every induced cycle in $G$ has length 3, and is \Index{co-chordal} if the complement graph $\overline{G}$ is chordal. It follows from Fr\"oberg's classification of edge ideals with linear resolutions [14] that $\reg (R/I(G))\le  1$ if and only if $\overline{G}$ is co-chordal. 
Due to Conjecture \ref{sregConjecture}, it is natural to prove the following result that is partially parallel to Fr\"oberg's classification.

\begin{Theorem}
  \label{cochordal1}
  It $G$ is a co-chordal graph with at least one edge, then $\sreg(S/I(G)) \le 1$.
\end{Theorem}

\begin{proof}
  As observed by \cite[Theorem 2.8]{MR2943752}, the paper \cite{MR0130190} has actually showed that a simple graph is chordal if and only if every induced subgraph of it has a simplicial vertex. Thus, we use a induction on the number of vertices of the complement graph $\overline{G}$ and apply Lemma \ref{simplicial-vertex}.
\end{proof}

The \Index{co-chordal cover number}, denoted by $\cochord(G)$, is the minimum number of co-chordal subgraphs required to cover the edges of $G$.  
Like \cite[Lemma 1]{arXiv:1009.2756}, we have the following result

\begin{theorem}
  \label{cochord-thm}
  For any simple graph $G$, we have $\sreg (S/I(G)) \le \cochord(G)$.
\end{theorem}

\begin{proof}
  We cover the graph $G$ by co-chordal subgraphs $G_1,\dots,G_c$ where $c=\cochord(G)$ and let $I_i=I(G_i)$. Now the result follows directly from Theorem \ref{cochordal1} and Lemma \ref{KM-type}(a).
\end{proof}

 An \Index{independent set} of $G$ is a subset of pairwise non-adjacent vertices. 
And $G$ is a \Index{split graph} if $V(G)$ can be partitioned into a clique and an (induced) independent set. Split graphs are both chordal and co-chordal. Covering the edges of $G$ with split graphs allows us to have

\begin{corollary}
  If $G$ is a simple graph such that $V(G)$ can be partitioned into an (induced) independent set $J_0$ together with $s$ cliques $J_1,\dots,J_s$, then $\sreg(S/I(G))\le s$.
\end{corollary}

\begin{corollary}
  If $G$ is a simple graph such that $A\subset V(G)$ induces a clique, then
  \[
  \sreg(S/I(G)) \le \sreg(S/I(G\setminus A)) +1.
  \]
\end{corollary}

\begin{corollary}
  \label{minimax-matching}
  If $G$ be a simple graph and $\beta(G)$ is the minimax matching number of $G$, then $\sreg(S/I(G))\le \beta(G)$.
\end{corollary}

The proofs for the above three corollaries are similar to those for \cite[Theorems 2, 3 and 11]{arXiv:1009.2756} and we will not repeat here. 

\begin{remark}
It is not difficult to see that the invariant $\beta(G)$ in Corollary \ref{minimax-matching} is bounded above by $\cosize(I(\calC))$, thus the result we established in Theorem \ref{cochord-thm} is better in this situation.
\end{remark}

Here are some additional applications of Theorem \ref{cochord-thm}.
Recall that an \Index{interval graph} is an intersection graph of a family of intervals (closed, open or half-open) on the real line. Interval graphs are chordal. A \Index{co-interval graph} is the complement of an interval graph. The \Index{boxicity} of a graph $G$, denoted $\boxicity(G)$, is the cardinality of the smallest edge covering of $\overline{G}$ by co-interval spanning subgraphs (by an equivalent definition by \cite[Corollary 3.1]{MR0712922}).
Thus, $\cochord(G)\le \boxicity(\overline{G})$. As a corollary to Theorem \ref{cochord-thm}, we have

\begin{corollary}
If $G$ is a simple graph, then  $\sreg(R/I(G))\le \boxicity(\overline{G})$.
\end{corollary}

\begin{corollary}
  If $G$ is a co-planar graph, namely if $\overline{G}$ is planar, then $\sreg(S/I(G))\le 3$.
\end{corollary}

\begin{proof}
  It follows directly from the fact  that $\boxicity(\overline{G})\le 3$ (\cite{MR0830590}).
\end{proof}

\begin{remark}
  We are not sure whether the upper bound of the above inequality can be achieved as in \cite[Proposition 18]{arXiv:1009.2756}. Notice that  for the special case when the graph $G=3K_2$, it is clear that $\overline{G}$ is the 1-skeleton of the Octahedron, hence planar. However,  $\sreg(S/I(G))=2$ by Corollary \ref{mk2}.
\end{remark}

\begin{remark}
  When $G$ is a simple graph, an \Index{induced matching} in $G$ is a matching which forms an induced subgraph of $G$ and that $\indmatch(G)$ denotes the number of edges in a largest induced matching. We have $\reg(S/I(G))\ge \indmatch (G)$ by \cite[Lemma 2.2]{MR2209703}. Unfortunately, we don't have $\sreg(S/I(G))\ge \indmatch(G)$. For instance, $\sreg(S/I(2K_2))=1<\indmatch(2K_2)=2$.
\end{remark}

\subsection{Clutters}
In this subsection, we will consider an upper bound of Stanley regularity in terms of combinatorial data from general clutters.  To be more specific, let $\calC$ be a clutter. Then a \Index{2-collage} for $\calC$, as defined in \cite{arXiv:1301.6779}, is a subset $C$ of edges with the property that for each $e\in E(\calC)$ we can delete a vertex $v$ so that $e\setminus\Set{v}$ is contained in some edge of $C$. In particular, when $\calC$ is a uniform clutter, the condition for $C$ to be a 2-collage is equivalent to saying that for any edge $e$ not in $C$, there is an edge $f\in C$ such that the cardinality of the symmetric difference of $e$ and $f$ is 2.

\begin{lemma}
  \label{2collage}
  If  $\Set{e_1}$ is a 2-collage for the clutter $\calC$, then $\sreg(S/I(\calC))\le |e_1|-1$.
\end{lemma}

\begin{proof}
  It suffices to show that $\sdepth_S(I^\vee)\ge n-|e_1|+1$ where $I=I(\calC)$. Suppose the edge set is $E(\calC)=\Set{e_i:1\le i \le s}$. For each edge $e_i$ of $\calC$, there is a monomial prime ideal $P_i=\braket{x_j:x_j\in e_i}\subset S$. Now $I^\vee =\bigcap_{i=1}^s P_i$. Without loss of generality, we may assume that $\sum_{i=1}^s P_i=\braket{x_1,\dots,x_n}$ and $P_1=\braket{x_1,\dots,x_r}$ with $1\le r \le n-1$.

  We will use $Y=\Set{x_1,\dots,x_r}$ as the splitting set in Construction \ref{splitting} and define the rings $S'$ and $S''$ accordingly.
 Now, $I^\vee=\bigoplus_{\tau \subset [s]} I_\tau$ is a decomposition of $I^\vee$ as a direct sum of $\ZZ^n$-graded $\KK$-subspaces of $I^\vee$.

  When $\tau=\emptyset$, $\sdepth_S(I_\emptyset)=\sdepth_{S'}(I^\vee\cap S')+n-r\ge 1+n-r$ by Lemma \ref{basic-facts}.

  When $\tau\ne \emptyset$, we can write $I_\tau$ as $I_\tau=H_\tau\otimes_\KK L_\tau$ where
  \[
    H_\tau=\frac{\bigcap_{j\notin \tau}P_j'+\sum_{j\in\tau}P_j'}{\sum_{j\in\tau}P_j'}
    \]
    and $L_\tau=\bigcap_{j\in \tau}P_j''$. Since $e_1$ is a 2-collage for $\calC$, $P_1''=0$ and each $P_i''$ is principal for $2\le i \le s$. In particular, if $L_\tau$ is nonzero, then it is principal and $\sdepth_{S''}L_\tau=\dim(S'')=n-r$. Meanwhile, as in the proof of Theorem \ref{LM}, we know $\sdepth_{S'}(H_\tau)\ge 1$. Consequently, if $I_\tau\ne 0$, then $\sdepth_S(I_\tau)\ge \sdepth_{S'}(H_\tau)+\sdepth_{S''}(L_\tau)\ge 1+n-r$ by \cite[Lemma 1.2]{MR2777680}.

  Now $\sdepth_S(I^\vee)\ge \min\Set{\sdepth_S(I_\tau):\tau\subsetneq [s] \text{ and } I_\tau\ne 0}\ge 1+n-r=1+n-|e_1|$.
\end{proof}

\begin{theorem}
  \label{collages}
  If $\Set{e_1,\dots,e_s}$ is a $2$-collage in the clutter $\calC$, then 
  \[
  \sreg(S/I(\calC))\le \sum_{i=1}^s (|e_i|-1).
  \]
\end{theorem}

\begin{proof}
  As in the proof for \cite[Theorem 1.2]{arXiv:1301.6779}, for each edge $e_i$ in the assumption, we set $H_i$ to be the clutter consisting of all edges $e$ with $e\setminus \Set{v}\subset e_i$ for some vertex $v$. Now $E(\calC)=\bigcup_{i=1}^s E(H_i)$ and each $H_i$ satisfies the condition of Lemma \ref{2collage}.  Now, we apply Lemma \ref{KM-type}(a).
\end{proof}


\end{document}